\documentclass{article}
\usepackage[utf8]{inputenc}
 
\title{Matrices, Bratteli Diagrams and Hopf-Galois Extensions}
\author{Ghaliah Alhamzi$^1$ \& Edwin Beggs$^2$ }

\date{%
    $^1$Department of Mathematics and Statistics, College of Science, \\ Imam Mohammad Ibn Saud Islamic University (IMSIU), Riyadh, Saudi Arabia \\%
    $^2$College of Science, Swansea University, Wales 
    \\[2ex]%
    \today
}

\usepackage{yfonts}
\usepackage{amsmath}
\usepackage{amsfonts}
\usepackage{amssymb}
\usepackage{amsthm}
\usepackage{mathrsfs}
\usepackage{bm}
\usepackage[pdftex]{graphicx,color}
\usepackage{graphics,color}
\usepackage{epstopdf}
\usepackage{graphicx}
\usepackage{textcomp}
\usepackage{tkz-berge}
\usepackage{pgfplots, pgfplotstable}
\usepackage[font=small,labelfont=bf,labelsep=colon]{caption}
\usepackage{pdfpages}
\usepackage{fullpage}
\usepackage{calligra} 
\usepackage[all]{xypic}

\newtheorem{theorem}{Theorem}[section]

\newtheorem{prop}[theorem]{Proposition}

\theoremstyle{remark}

\newtheorem{definition}[theorem]{Definition}

\theoremstyle{definition}

\newcommand{\id}{\mathrm{id}}

\newcommand{\extd}{\mathrm{d}}
\newcommand{\tens}{\mathop{\otimes}}

\newcommand{\C}{\mathbb{C}}

\newcommand{\Z}{\mathbb{Z}}

\newcommand{\mathsym}[1]{{}}
\newcommand{\unicode}[1]{{}}
\newcommand{\h}{{\scriptstyle\frac{1}{2}}}

\newcommand{\can}{\mathrm{ can}}
\newcommand{\ver}{\mathrm{ ver}}
\newcommand{\inc}{\mathrm{inc}}
\newcommand{\inv}{\mathrm{inv}}
\newcommand{\rep}{\mathrm{rep}}
\newcommand{\conv}{{\odot}}

\begin{document}
 
\maketitle

 \abstract{We show that the matrix embeddings in Bratteli diagrams are iterated direct sums of Hopf-Galois extensions (quantum principle bundles) for certain abelian groups. The corresponding strong universal connections are computed. We show that $ M_{n}(\C)$ is a trivial quantum principle bundle for the Hopf algebra $ \C[\Z_{n} \times \Z_{n}] $. We conclude with an application relating known calculi on groups to calculi on matrices.}
\section{Introduction}
We suppose that $ P $ is a unital algebra and that $ H $ is a Hopf algebra over $\C$. We write the coproduct of $H$ as
$\Delta(h)=h_{(1)}\tens h_{(2)}$ using the Sweedler notation. The idea of using a Hopf algebra in place of a group in a principal bundle 
was given in \cite{Sch:pri}. Independently in \cite{BrzMa:gau,BrzMa:cla} this was described in terms of differential calculi.  For the connection between classical Galois theory and 
Hopf-Galois extensions, see \cite{ChaSwe,KreTak,montGal}.

\begin{definition}
	For an algebra $ P $ and a Hopf algebra $ H $, we say that a right coaction $ \Delta _{R}: P\rightarrow P \tens H $ makes $ P $ into a right $ H $-comodule algebra if $ \Delta_{R} $ is an algebra map, i.e.\ for $ p,q \in P $ 
	\begin{eqnarray}
	(pq)_{[0]} \tens 	(pq)_{[1]} =p_{[0]} q_{[0]} \tens p_{[1]} q_{[1]}\ ,\quad \Delta_R(1)=1\tens 1
	\end{eqnarray}
	where we write the coaction as $ \Delta_{R}=p_{[0]} \tens p_{[1]}    $.
\end{definition}

This means that the invariants $ A=P^{\mathrm{coH}} =\{p\in P:\Delta_R(p)=p\tens 1\}$ is a subalgebra of $P$. 

\begin{definition}
	Let $ P $ be a right $ H $-comodule algebra. $ P $ is a Hopf-Galois extension of $ A=P^{\mathrm{coH}} $ if the canonical map 
	$\mathrm{ver}^{\sharp}:P \tens_{A} P\rightarrow P \tens H$ 	is a bijection, where 
	\begin{eqnarray}
\mathrm{ver}^{\sharp}(p\tens q) = pq _{[0]}  \tens q_{[1]}  \ .
	\end{eqnarray}
\end{definition}
The idea of Hopf-Galois extension can also be described as a quantum principle bundle for the case of universal differential calculi \cite{BMbook}. We shall return to this later where we give an application to differential calculi on the matrices. If $ P $ is a Hopf-Galois extension then there are elements of $ P \tens_{A} P $ mapping to $ 1 \tens h $ for all $ h \in H $. For many practical purposes we seek an element $ h^{(1)} \tens h^{(2)} \in P \tens P $ (not $ \tens_{A}$) mapping to $ 1 \tens h $ under the canonical map. This is no longer unique, but we ask whether there is a function $ \omega^{\sharp}: H\rightarrow P \tens P $ given by $ \omega^{\sharp}(h) =h^{(1)} \tens h^{(2)}  $ such that
$\omega^{\sharp}(1)=1\tens 1$ and
\begin{align} \label{equns5}
h^{(1)}    \tens h^{(2)}{}_{[0]}    \tens    h^{(2)} {}_{[1]}     &=    h_{(1)}{}^{(1)}    \tens h_{(1)}{}^{(2)}    \tens    h_{(2)}  \ ,\cr
h^{(1)} {}_{[0]}     \tens h^{(1)} {}_{[1]}   \tens    h^{(2)}    &=    h_{(2)}{}^{(1)}    \tens S h_{(1)}  \tens    h_{(2)} {}^{(2)}   
\end{align}
in which case we say that $  \omega^{\sharp} $ is a strong universal connection (this name was given in \cite{Haj}). In \cite{BBstrong} it was shown that if $ H $ has a normalised integral then $ \omega^{\sharp} $ always exists, using the next result, for which we provide a framework of the proof as we shall require it later.
\begin{theorem}\label{integralHopfGalois} Suppose that $H$ has normalised  left-integral $\int$ and bijective antipode, and that $P$ is a right $H$-comodule algebra with $\ver^\sharp$ surjective. Then $(P,H,\Delta_R)$ is a universal quantum  principal bundle and admits a strong connection. To show this begin with a linear map $h\mapsto h^{(1)}\tens h^{(2)}\in P\tens P$ so that $1\mapsto 1\tens 1$ and
$\ver^\sharp(h^{(1)}\tens h^{(2)})=1\tens h\in P\tens H$, but not necessarily satisfying (\ref{equns5}).  Now define $b: H \tens H \rightarrow \C$ by $b (h,g)=\int (h S g)$, and then
$a_{R}: P \tens H \rightarrow P$ and $ a_{L}: H \tens  P \rightarrow P$ by
$a_{R} (p \tens h)= p_{[0]} b (p_{[1]}, h)$ and $ a_{L}( h\tens p )=b (h, S^{-1} p_{[1]})  p_{[0]} $. Then we have a strong universal connection
\begin{eqnarray} 
\omega^{\sharp}(h)= a_{L}(h_{(1)} \tens h_{(2)}{} ^{(1)})  \tens a_{R} (   h_{(2)}{}^{(2)}\tens h_{(3)} )
\end{eqnarray}
\end{theorem}

In 1972 Ola Bratteli introduced graphs for describing certain classes of $ C^{*} $- algebras in terms of limits of direct sums of matrices \cite{Brat}. This is a graph split into levels, and an example of one level in a Bratteli digram is
\begin{align} \label{eq47}
 \UseComputerModernTips
M_{1} \oplus M_{2} = 
\left\lbrace   \vcenter{\hbox{
		\xymatrix{
			M_{1}  {\bullet}  \ar@{=}[rd] \ar@{-}[r] &{\bullet} M_{1}
			\\
			M_{2} {\bullet}  \ar@{-}[r]  & {\bullet} M_{4} 
}  }} \right\rbrace 
=M_{1} \oplus M_{4}=P
\end{align}  
which represents the map 
\begin{eqnarray*}
(a)\oplus \left(\begin{array}{cc}
b&  c\\
d& e
\end{array}
 \right) \mapsto (a) \oplus \left(\begin{array}{cccc}
a& 0& 0& 0 \\
0& a& 0& 0 \\
0& 0& b& c \\
0& 0& d& e
\end{array}
\right) \ .
\end{eqnarray*}
In Section \ref{sec1} we will show that such diagrams do not necessarily give Hopf-Galois extensions. However, we can split the level on the diagram into smaller pieces, each of which is a direct sum of Hopf-Galois extensions. For example, we rewrite (\ref{eq47}) as a composition of three stages
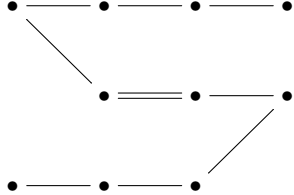
\begin{align*} \label{eq48} 
\UseComputerModernTips
\xymatrix{
  {\bullet}  \ar@{-}[r]  \ar@{-}[rd]&{\bullet}  \ar@{-}[r] &{\bullet}  \ar@{-}[r]&{\bullet} 
	\\
& {\bullet}  \ar@{=}[r]&{\bullet} \ar@{-}[r]&{\bullet} \\
 {\bullet} \ar@{-}[r]& {\bullet}   \ar@{-}[r]&{\bullet}  \ar@{-}[ru]&
}
\end{align*} 
and we refer to these as (from right to left) Case 1, Case 2 and Case 3. In each of these cases we shall exhibit a strong universal connection map. For our purposes it is sufficient to consider $ H=\C[G] $, the complex valued functions on a finite group $ G $. Then $ H $ has basis $ \delta_{g} $, the function which is $ 1 $ at $ g \in G $ and zero elsewhere. The Hopf algebra operators are 
\begin{align*}
\delta_{x} \delta_{y}&= \begin{cases}
\delta_{x} & \mbox{if}\ x=y \\
0 & \mbox{if} \ x \neq y
\end{cases} \quad \text{ and} \quad \Delta \delta_{g} =\sum_{  x,y : xy=g}  \delta_{x} \tens \delta_{y}\cr
1&= \sum_{x \in G} \delta_{x}\ , \quad \epsilon(\delta_{x})=\delta_{x,e}\ , \quad  S (\delta_{x})=\delta_{x^{-1}}\ .
\end{align*}
The idea of a trivial quantum principle bundle was set down in \cite{Ma:nonab}. In Section \ref{secTQPB} we show that $ M_{n}(\C) $ is a trivial quantum principle bundle for the Hopf algebra $ \C[\Z_{n} \times \Z_{n}] $. Note that in \cite{BrzMa:cmp} it was shown that $ M_{n}(\C) $ was an algebra factorisation of two copies of $ \C[\Z_{n}] $ satisfying a Galois condition. We conclude by an application relating differential calculi on $ M_{n}(\C) $ to differential calculi on $ \C[\Z_{n} \times \Z_{n}] $.  

The example of differential calculi shows the Hopf-Galois extensions described here have applications, and in general quantum principle bundles are an expanding era of interest in noncommutative geometry. In algebraic topology iterated fibrations (often called towers of fibrations) often occur, e.g.\ Postnikov systems or Postnikov towers \cite{PostTower}. Here the mere existence of these iterated fibrations is very useful.
\section{A Bratteli Diagram which is not a Hopf-Galois Extension } \label{sec1}
We shall show that the Bratteli diagram 
\begin{align}  \UseComputerModernTips
M_{1} \oplus M_{1} = 
\left\lbrace   \vcenter{\hbox{
		\xymatrix{
			M_{1}  {\bullet}  \ar@{-}[rd] \ar@{-}[r] &{\bullet} M_{1}
			\\
			M_{1} {\bullet}  \ar@{-}[r]  & {\bullet} M_{2} 
}  }} \right\rbrace 
=M_{1} \oplus M_{2}=P
\end{align} 
does not give an inclusion coming from a Hopf-Galois extension. In terms of matrices this is 
\begin{eqnarray*}
(a)\oplus (b)\mapsto (a) \oplus \left(\begin{array}{cc}
a& 0 \\
0& b
\end{array}
\right) 
\end{eqnarray*}
$ P $ has a linear basis $ (1) \oplus 0 $ and $ (0) \oplus E_{ij} $ for $ i,j \in \{ 0,1\} $ and $ A $ has a linear basis $ (1) \oplus E_{11} $ and $ (0) \oplus E_{22} $. Note that multiplication by elements of $ A $ simply scales each of the given basis vectors in $ P $ by a number. Thus to find $ P \tens_{A} P $ we only have to consider the first element of $ P $ to be a basis element. Note that in $ P \tens_{A} P $
\begin{eqnarray*}
((1) \oplus 0 ) \tens (\alpha \oplus \beta)=((1) \oplus 0 ) ((1) \oplus E_{11})  \tens (\alpha \oplus \beta)=((1) \oplus 0 )\tens  ((1) \alpha \oplus E_{11} \beta)
\end{eqnarray*}
so we have
\begin{eqnarray*}
((1) \oplus 0 )\tens  ((0) \oplus E_{2j})=0 \ .
\end{eqnarray*}
Next
\begin{eqnarray*}
((0) \oplus E_{i1})  \tens (\alpha \oplus \beta)=((0) \oplus E_{i1} ) ((1) \oplus E_{11})  \tens (\alpha \oplus \beta) =((0) \oplus E_{i1} ) (\alpha \oplus E_{i1} \beta)  
\end{eqnarray*}
so we have 
\begin{eqnarray*}
((0) \oplus E_{i1})\tens  ((0) \oplus E_{2j})=0 \ .
\end{eqnarray*}
Next
\begin{eqnarray*}
( (0) \oplus E_{i2} )\tens (\alpha \oplus \beta) =((0) \oplus E_{i2} ) ((0) \oplus E_{22}) \tens (\alpha \oplus \beta) =((0) \oplus E_{i2} )\tens  ((0) \oplus E_{22} \beta) 
\end{eqnarray*}
so we have  
\begin{eqnarray*}
( (0) \oplus E_{i2} )\tens ((1)\oplus 0) =((0) \oplus E_{i2} )\tens  ((0) \oplus E_{1j}) =0 \ .
\end{eqnarray*}
We have shown that $ 12 $ tensor products of the basis of $ P $ with itself disappear, making $ P \tens_{A} P $ $ 13 $ dimensional. If this was  a Hopf-Galois extension this would have to be $ \dim P  \times \dim H$, which would be a multiple of $ 5 $.
\section{Combining Matrices on Block Diagonals}
We consider two cases of subalgebras $ A $ of $P= M_{m}(\C) $ and also a subalgebra of the direct sum of matrix algebras,  and show that they form quantum principle bundles. In the following section we count matrices from entry $ 0,0 $ in the top left, and use mod $ m $ arithmetic for the rows and columns of $ M_{m}(\C) $. We follow on from the previous section by calculating $ P \tens_{A} P $ in these three cases.
\subsection*{Case 1:}
We choose block decompositions of $ M_{m}(\C) $ with rows and columns being divided into intervals of nonzero length $ l_{0}, l_{1}, \dots, l_{n-1}  $ where $ l_{0}+l_{1}+\ldots +l_{n-1}=m $. Let $ A $ be the image of the nonzero diagonal embedding
\begin{eqnarray}\label{eq30}
M_{l_{0}}(\C) \oplus M_{l_{1}}(\C)  \oplus \ldots\oplus  M_{l_{n-1}}(\C)  \longrightarrow M_{m}(\C)\ .
\end{eqnarray}
For row or column $j$ we take $(j)\in\Z_n$ to be the block to which row or column $j$ belongs.
Thus for $ l_{0}=2, l_{1}=1 , m=3$ we have $ \left(\begin{array}{cc|c}
a& b &0 \\ 
c& d & 0 \\ \hline
0&0  &  e
\end{array}  \right)   \in A $ and $(0)=0$, $(1)=0$, $(2)=1$.

\begin{prop}
$ M_{m} (\C) \tens_{A} M_{m}(\C) $ where $ A $ is the image of (\ref{eq30}) is given by the isomorphism of $ M_{m}(\C)$-$M_{m}(\C)$ bimodules $ \UseComputerModernTips
\xymatrix{
	M_{m}(\C)\tens_{A} M_{m}(\C) \ar[r]^-{T} & M_{m} (\C) \tens \C [\Z_{n}]  }$ which is given by
\begin{eqnarray*}
T(E_{ij}\tens E_{ab})=E_{i b} \delta_{ja}\tens\delta_{(j)}
\end{eqnarray*}
\end{prop}
\begin{proof}
	First we get, as $ E_{aa}\in A $,
\begin{eqnarray*}
	E_{ij}\tens E_{ab}=E_{i j} \tens E_{aa} E_{ab}  = E_{i j} E_{aa}\tens  E_{ab} = \delta_{ja} E_{ij}\tens E_{ab}\ .
\end{eqnarray*}
Next if $ (k)=(j) $ then $ E_{kj}\in A $ so 
\begin{eqnarray*}
E_{ij}\tens E_{jr}= E_{ik}E_{kj}\tens E_{jr}= E_{ik}\tens E_{kj}E_{jr}= E_{ik}\tens E_{kr}\ .
\end{eqnarray*}
\end{proof}
\subsection*{Case 2:}
We take the embedding $ M_{k}(\C)\longrightarrow M_{m}(\C) $ where $ m=nk $ sending the matrix $ x $ to the block diagonal matrix with  all diagonal blocks being $ x $, and let  $ A $ be the image. E.g.\  for $ n=3 $ we have
\begin{equation}\label{eq31}
x\mapsto \left( \begin{array}{ccc}
x&0  &0  \\ 
0&x  &0  \\ 
0&0  &x 
\end{array} \right) \ .
\end{equation}
\begin{prop}
	$ M_{m} (\C) \tens_{A} M_{m} (\C)$ where $ A $ is the image of (\ref{eq31}) is given by the isomorphism of $ M_{m}(\C)$-$M_{m}(\C)$ bimodules $ R: M_{m}(\C)\tens_{A} M_{m}(\C) \longrightarrow M_{m}(\C) \tens \C [\Z_{n}]  \tens \C [\Z_{n}] $ which is given by	
	\begin{equation*}
	R(E_{ij}\tens E_{ab})=\begin{cases}
	E_{i b}\tens \delta_{r}\tens\delta_{(a)}& \mbox{if}\  j=a+kr \mod m,\  \text{for}\  0\leq r< n\\
	0 & \mbox{otherwise} \ .
	\end{cases}
	\end{equation*}
\end{prop}
\begin{proof}
	Setting $ y=\displaystyle\sum_{0\leq r < n} E_{a+kr, a+kr}\in A $, we have
	\begin{eqnarray*}
	E_{ij} \tens E_{ab}=E_{ij}\tens y E_{ab}=E_{ij} y \tens E_{ab}
	\end{eqnarray*}
$ E_{ij} \tens E_{ab}=0 $ unless $ j=a+kr $ mod $ m $ for some $ 0\leq r < n $. Now $ A $ has linear basis $ \displaystyle \sum_{0\leq r < n} E_{a+kr, d+kr} =G_{ad}$ for $ (a)=(d) $, i.e. $ a $ and $ d $ in the same block. Then if $ j=a+kr $ element of $ \Z_{m} $
	\begin{eqnarray*}
	E_{ij}\tens E_{ab}=E_{ij}\tens G_{ad}E_{db} =E_{ij} G_{ad}\tens E_{db}	=E_{ij} E_{a+kr, d+kr}\tens E_{db}	=E_{i, d+kr}\tens E_{db}\ .
	\end{eqnarray*}
So the only nonzero tensor product of the $ E_{ij} $ has the following relation, where $ (a)=(d) $, 
\begin{eqnarray*}
E_{i,a+kr}\tens E_{ab}=  E_{i,d+kr}\tens E_{db}\ .
\end{eqnarray*} 
Thus we have $ r $, $(a) $ and the product $ E_{ib} $ are the same of both sides of the relation, so $ R $ is well defined.
\end{proof}
\subsection*{Case 3:} 
For unital algebra $ B $, consider the replication map $\rep: B \rightarrow B ^{\oplus n}=P $ given by 
\begin{eqnarray*}
\mathrm{rep}(b)= b\oplus \dots\oplus b 
\end{eqnarray*}
and call its image $ A $. For $  s\in \Z_{n} $ we label $ b_{,s} $ as the element of $ B^{\oplus n} $ which is $ b \in B $ is the $ s $th component and $ 0 $ in the other positions. 
The elements of $ A $ are of the form $ a=\displaystyle\sum_{  s} b_{,s}$. Now for $ c \in B $
\begin{eqnarray*}
c_{,r} \tens b_{,t} =c_{,r}   \tens a\cdot  1_{,t} =c_{,r}  a  \tens 1_{,t} =(cb)_{,r}  \tens 1_{t}
\end{eqnarray*}
so $ P\tens_{A} P $ has an isomorphism $ u: P \tens_{A} P\rightarrow P \tens C (\Z_{n}) $, where $ t \in \Z_{n} $
\begin{eqnarray*}
u(c_{,r} \tens b_{,t})= (cb)_{,r}  \tens \delta_{t}\ .
\end{eqnarray*}

\section{Block Matrices and Quantum Principle Bundles}\label{sec2}
In this section we will show that the three cases in the previous section are actually examples of quantum principle bundles. In the definition of Hopf-Galois extension we only need the case where $  H=\C[G]$. If $ G $ acts on the algebra $ P $ on the left by algebra maps, i.e. $ g\triangleright (pq)=(g\triangleright  p)(g\triangleright  q) $ then we have a right $ \C[G]$ comodule algebra $\Delta_{R }:P \rightarrow P \tens H $ by
\begin{equation}
\Delta_{R }(p)=\sum_{g\in G} g\triangleright p \tens \delta_{g}\ .
\end{equation}
We set $ R_{n} $ to be the group of complex $ n $th roots of unity, and recall that we label matrices from row and column zero.
\subsection*{Case 1:}
We define the group
\begin{equation}
G=\left\lbrace g_{w}=\left( \begin{array}{c|c|c|c}
1&  &  &  \\\hline
&\omega  &  &  \\ \hline
&  & \ddots &  \\ \hline
&  &  & \omega^{n-1}
\end{array}
 \right):\quad  \omega \in R_{n}  \right\rbrace  \subset   G L_{m}(\C)
\end{equation}
using the blocks of length $ l_{0}, \cdots l_{n-1} $. Now $ G $ acts on $ P=M_{m}(\C) $ by algebra maps $ g_{\omega} \triangleright x= g_{\omega} x g_{\omega}^{-1}$. For $ x $ purely in the $ st $ block for $ s,t \in \Z_{n} $ we have
\begin{eqnarray*}
g_{\omega} \triangleright x=\omega ^{s-t} x
\end{eqnarray*}
so the fixed points of the $ G $ action are precisely the block diagonal subalgebra $ A $. The canonical map is 
\begin{eqnarray*}
\can (E_{ij}\tens E_{ab})=\delta_{ja} \sum_{\omega \in R_{n}} E_{ib} \,\omega^{(a)-(b)} \tens \delta_{\omega} \quad \text{ where} \ a,b,i,j \in \Z_{m} \ .
\end{eqnarray*}
If the canonical map is surjective, then it is automatically injective since the dimension of $ M_{m}(\C) \tens_{A} M_{m}(\C) $ is the same as $ M_{m} (\C) \tens H $ by the previous section. Now for $  \xi \in R_{n} $
\begin{align*}
\can \big(\xi^{(i)-(j)}E_{ij}\tens E_{ji}\big)=\sum_{\omega \in R_{n}} E_{ii} \,\bigg(\frac{\omega}{\xi} \bigg)^{(j)-(i)} \tens \delta_{\omega}
\end{align*}
so
\begin{align*}
 \sum_{j\in \Z_{m} } \frac{1}{l_{(j)}}\xi^{(i)-(j)} \can  \big( E_{ij}\tens E_{ji}\big)=\sum_{q \in \Z_{n},\omega \in R_{n}} E_{ii} \,\bigg(\frac{\omega}{\xi} \bigg)^{q-(i)} \tens \delta_{\omega} 
\end{align*}
and by the formula for the sum of a geometric progression this sum is zero unless $ \dfrac{\omega}{\xi}=1 $, so 
\begin{eqnarray} \label{eq32}
 \can   \big(   \sum_{j,i} \frac{1}{l^{(j)}}\xi^{(i)-(j)}  E_{ij}\tens E_{ji}\big)=n\,\sum_{i} E_{ii} \tens \delta_{\xi}=n \cdot I_{m}\tens \delta_{\xi}
\end{eqnarray}
since the canonical map is  a left $ P $-module map we see that it is surjective, and we have a quantum principle bundle. We have proved the following Proposition
\begin{prop} \label{Prop 3}
	For $ \xi \neq 1 $ set 
	\begin{eqnarray}
	\delta^{(1)}_{\xi} \tens \delta^{(2)}_{\xi}= \dfrac{1}{n}   \sum_{j,i} \frac{1}{l^{(j)}}\xi^{(i)-(j)}   E_{ij}\tens E_{ji}\ , \quad
\delta^{(1)}_{1} \tens \delta^{(2)}_{1}	=  I_{m} \tens I_{m} -\sum_{\xi \neq 1}   \delta^{(1)}_{\xi} \tens \delta^{(2)}_{\xi} \ .
	\end{eqnarray}
	Then $ 1^{(1)} \tens 1^{(2)}=I_{m}\tens I_{m} $ and for all $ \eta \in R_{n} $
	\begin{eqnarray}
	\can\big(\delta^{(1)}_{\eta} \tens \delta^{(2)}_{\eta}\big)=I_{m} \tens \delta_{\eta}\ .
	\end{eqnarray}
\end{prop}
\subsection*{Case 2:}
For $ \omega \in R_{n} $ and $ i\in \Z_{n} $, define using blocks of length $ k $
\begin{eqnarray*}
g_{i,\omega}=\left( \begin{array}{ccccccc}
0& \dots &0  &  1&  0&\dots & 0 \\
0& \dots & 0 & 0 &\omega  &\dots  & \dots\\
\colon& \colon & \colon &  \colon& \colon & \colon &\colon\\
\omega^{n-i}& 0 &0  &0  &0  &0  &0\\
\colon& \colon & \colon &  \colon& \colon & \colon &\colon\\
0& \dots &\omega^{n-1}  &0  &0  &\dots &0
\end{array}
\right) 
\end{eqnarray*}
where the original $ 1 $ (in fact $ I_{k} $) is in column $ i $ (counting from column $0  $). Note that 
\begin{eqnarray}
g_{i,\omega} g_{j,\eta}=\eta^{i}g_{i+j,\omega \eta}, \quad 
g_{i,\omega}^{-1}=\omega^{i}g_{-i, \frac{1}{\omega} }\ .
\end{eqnarray}
We take the group $ G $ of projective matrices $ G \subset P GL_{n}(\C) $ consisting of the $ g_{i, \omega} $, so we get $ G \cong \Z_{n} \times R_{n} $. Then define an action of $ G $ on $ M_{m}(\C) $ by $ (i,\omega) \triangleright x=g_{i, \omega} x g_{i, \omega} ^{-1} $
which is not dependent on a scale factor on the $ g_{i,\omega} $. We use $ F_{jt} $ for $ j,t \in \Z_{n} $ to denote the identity matrix in the $ jt$ block and zero elsewhere. Now
\begin{eqnarray*}
(i,\omega) \triangleright F_{jt}=\omega^{j-t} F_{j-i,t-i}, \quad \text{so }\ \Delta_{R} F_{jt}= \sum_{s,\omega} \omega^{j-t} F_{j-s, t-s} \tens \delta_{(s,\omega)}
\end{eqnarray*}
and the canonical map is 
\begin{eqnarray*}
\mathrm{can}(F_{ab}\tens F_{jt})= \sum_{\omega}\omega^{j-t} F_{a,t-j+b}\tens \delta_{(j-b, \omega)} 
\end{eqnarray*}
and a particular case of this is, by setting $ b=j-i $ and $ t=i+a $
\begin{eqnarray*}
\mathrm{can}(F_{a,j-i}\tens F_{j,i+a})= \sum_{\omega}\omega^{j-i-a} F_{a,a}\tens \delta_{(i, \omega)}\ .
\end{eqnarray*}
Now, using the sum of powers of a root of unity, 
\begin{align*}
\mathrm{can}  \big(\sum_{j}   \xi^{i-j}  (F_{a,j-i}\tens F_{j,i+a}) \big)= \sum_{\omega,j} \omega^{-a} (\frac{\omega}{\xi})^{j-i} F_{a,a}  \tens \delta_{(i, \omega)}
=n \, \xi^{-a} F_{a,a} \tens \delta_{(i, \xi)} 
\end{align*}
so
\begin{align}
\mathrm{can}  \big(\sum_{j,a} \dfrac{1}{n}   \xi^{i-j+a}  (F_{a,j-i}\tens F_{j,i+a}) \big)=\sum_{a} F_{a,a} \tens \delta_{(i, \xi)} 
=I_{m}\tens \delta_{(i, \xi)}\ .
\end{align}
Thus we have proved the following result 
\begin{prop}\label{Prop2}
If we define, for $ (i, \xi)\neq (0,1) $	
\begin{eqnarray}
\delta_{(i, \xi)}^{(1)}\tens \delta_{(i, \xi)}^{(2)}=\dfrac{1}{n}\sum_{j,a} \xi^{i-j+a} F_{a,j-i}\tens F_{j,i+a}\ , \quad \delta_{(0, 1)}^{(1)}\tens \delta_{(0, 1)}^{(2)}= I_{m}\tens I_{m} -\sum_{(i,\xi)\neq (0,1)}   \delta_{(i, \xi)}^{(1)}\tens \delta_{(i, \xi)} {} ^{(2)}\ .
\end{eqnarray}
	Then $ 1^{(1)} \tens 1^{(2)}=I_{m}\tens I_{m} $ and for all $ (j,\eta) \in  \Z_{n} \times R_{n}  $ 
		\begin{eqnarray}
	\can\big(\delta^{(1)}_{(j,\eta) } \tens \delta^{(2)}_{(j,\eta) }\big)=I_{m} \tens \delta_{(j,\eta) }\ .
	\end{eqnarray}
\end{prop}
\subsection*{Case 3:}
We use the group $ G= \Z_{n} $ acting on $ P= B^{\oplus n} $ by $ i\triangleright b_{,s}=b_{,s+i}\mod n$. Now 
\begin{eqnarray}
\can (1_{,t}   \tens 1_{,s})= 1_{,t} . \sum_{i \in \Z_{n}} 1_{,s+i} \tens \delta_{i}=1_{,t}  \tens \delta_{t-s}
\end{eqnarray}
so we have the following Proposition. 
\begin{prop}
If we define, for $ i \in \Z_{n} $
	\begin{eqnarray}\label{eq53}
\delta_{i}^{(1)} \otimes \delta_{i}^{(2)}= \sum_{t}  1_{,t}  \tens 1_{,t-i} \ .
	\end{eqnarray}
		Then $ 1^{(1)} \tens 1^{(2)}= \sum_{t}  1_{,t}  \tens \sum_{s}  1_{,s} $ and
	\begin{eqnarray}
	\can\big(\delta^{(1)}_{i} \tens \delta^{(2)}_{i }\big)= \sum_{t}  1_{,t}  \tens\delta_{i} \ .
	\end{eqnarray}
\end{prop}
\section{Strong Universal Connection}
We find the strong universal connections corresponding to the cases in the previous section, starting with the back maps $h\mapsto h^{(1)} \tens h^{(2)} $ given there. Note these have been defined so that $ 1^{(1) }\tens 1^{(2)}= 1 \tens 1 $. This uses Theorem \ref{integralHopfGalois} and normalised integral on $ \C[G]$ for $ G $ a finite group $ \int f= \dfrac{1}{|    G |}\displaystyle \sum_{g \in G} f(g)$.
\begin{prop}
	In Case 1
		\begin{eqnarray}
	\omega^{\sharp} (\delta_\eta )=	\dfrac{1}{n} \sum_{i,j\in \Z_{m}: (i)\neq (j)} \frac{1}{l^{(j)}} \eta^{(i)-(j)}  E_{ij}   \tens  E_{ji}+\dfrac{1}{n}I_{m} \tens I_{m}\ .
	\end{eqnarray}
\end{prop}

\begin{proof}
	By Theorem \ref{integralHopfGalois} and  Proposition \ref{Prop 3}
		\begin{align*}
	\omega^{\sharp} (\delta_\eta )&= \sum_{\omega,\xi} a_{L} (\delta _{\omega} \tens \delta _{\xi}^{(1)}) \tens a_{R}(\delta _{\xi}^{(2)} \tens \delta _{\eta \omega^{-1}\xi^{-1}})\cr 
	&= \sum_{\omega,\xi \neq 1}  a_{L} (\delta _{\omega} \tens \delta_{\xi}^{(1)} ) \tens  a_{R} (\delta_{\xi}^{(2)} \tens  \delta _{\eta \omega^{-1} \xi^{-1}}  ) + \sum_{\omega} a_{L} (\delta _{\omega} \tens I_{m})\tens  a_{R} (I_{m}\tens  \delta _{\eta \omega^{-1}}  )\cr
	& \quad- \sum_{\omega,\xi \neq 1}  a_{L} (\delta _{\omega} \tens \delta_{\xi}^{(1)} ) \tens  a_{R} (\delta_{\xi}^{(2)} \tens  \delta _{\eta \omega^{-1}}  ) \cr
	&= \sum_{\omega,\xi \neq 1}  a_{L} (\delta _{\omega} \tens \delta_{\xi}^{(1)} ) \tens  a_{R} (\delta_{\xi}^{(2)} \tens  (\delta _{\eta \omega^{-1} \xi^{-1}} -   \delta _{\eta \omega^{-1} } ) )+ \sum_{\omega} a_{L} (\delta _{\omega} \tens I_{m})\tens  a_{R} (I_{m}\tens  \delta _{\eta \omega^{-1}}  )\cr
	&=  \dfrac{1}{n} \sum_{\omega,\xi \neq 1,i,j}    \frac{1}{l^{(j)}}   \xi^{(i)-(j)} a_{L} (\delta _{\omega} \tens E_{ij} ) \tens a_{R}\big( E_{ji}  \tens  (\delta _{\eta \omega^{-1}\xi^{-1}} - \delta _{\eta \omega^{-1}}) \big )\cr
	&\quad + \sum_{\omega} a_{L} (\delta _{\omega} \tens I_{m})\tens  a_{R} (I_{m}\tens  \delta _{\eta \omega^{-1}}  )\ .
	\end{align*}
	Then  we  calculate  $ a_{R} $ and $ a_{L} $ for both sums by $ a_{R}(E_{i j}\tens \delta_{\xi} ) = \dfrac{1}{n} \xi^{(j)-(i)} E_{ij}$ and $	a_{L} (\delta_{\xi} \tens E_{i j}) =\dfrac{1}{n} \xi^{(i)-(j)}  E_{ij}   $ and so 
	\begin{align}\label{eq49}
		\omega^{\sharp} (\delta_\eta )	&=\dfrac{1}{n^{3}} \sum_{\omega,\xi \neq1,i,j} \frac{1}{l^{(j)}}\xi^{(i)-(j)} \omega^{(i)-(j)}  E_{ij}   \tens  E_{ji}  \big((\eta \omega^{-1} \xi^{-1})^{(i)-(j)}- (\eta \omega^{-1} )^{(i)-(j)} \big )\cr 
	&\quad+\sum_{\omega,i} \dfrac{1}{n}E_{ii}\tens\sum_{i} \dfrac{1}{n}E_{ii}\cr
	&=\dfrac{1}{n^{3}} \sum_{\omega,\xi \neq1,i,j} \frac{1}{l^{(j)}}\xi^{(i)-(j)} \eta^{(i)-(j)} \big(\xi^{(j)-(i)} -1\big) E_{ij}   \tens  E_{ji}+\sum_{\omega} \dfrac{1}{n^{2}}I_{m} \tens I_{m} \cr
	&=\dfrac{1}{n^{2}} \sum_{\xi \neq1,i,j} \frac{1}{l^{(j)}} \eta^{(i)-(j)}  E_{ij}   \tens  E_{ji}\big(1-\xi^{(i)-(j)} \big)+ \dfrac{1}{n}I_{m} \tens I_{m}\ .
	\end{align}
	Now for the first term if $ (i)-(j) \neq 0 $ we get 
	\begin{eqnarray}\label{eq50}
	\sum_{\xi} \xi^{(i)-(j) }=0=1+\sum_{\xi\neq 1} \xi^{(i)-(j) }
	\end{eqnarray}
	we split the last equation in (\ref{eq49}) into an $ (i)=(j) $ part (the summand vanishes), and $ (i)\neq(j) $ part where summing over $ \xi $ and using (\ref{eq50}) gives
	\begin{eqnarray*}
	\omega^{\sharp} (\delta_\eta )=	\dfrac{1}{n^{2}} \sum_{i,j: (i)\neq (j)} \frac{1}{l^{(j)}} \eta^{(i)-(j)}  E_{ij}   \tens  E_{ji}\big(n-1+1 \big)=	\dfrac{1}{n} \sum_{i,j: (i)\neq (j)} \frac{1}{l^{(j)}} \eta^{(i)-(j)}  E_{ij}   \tens  E_{ji}\ .
	\end{eqnarray*}
\end{proof}

\begin{prop}
		In Case 2
	\begin{eqnarray}\label{eq42}
	\omega^{\sharp} (\delta_{(k,\eta) })=	\dfrac{1}{n} 	\sum_{  b,s\in\Z_{m}} 
	\eta^{b-s}  F_{b,s}  \tens F_{s+k,b+k} -   \dfrac{1}{n^{2}}   	\sum_{b, i \in \Z_{m}} F_{b,b}  \tens  F_{i+k+b,i+k+b}  \ .
	\end{eqnarray}
\end{prop}
\begin{proof}
	By Proposition \ref{Prop2}
\begin{align*}
\omega^{\sharp} (\delta_{(k,\eta) })&=\sum_{(i,p,\omega,\xi)} a_{L} \big(\delta _{(p,\omega)} \tens \delta _{(i,\xi)}^{(1)} \big) \tens a_{R}\big(\delta _{(i,\xi)}^{(2)} \tens \delta _{(k-p-i,\eta \omega^{-1}\xi^{-1})}\big)\cr 
&= \sum_{ (p,\omega),  (i,\xi) \neq(0, 1)} a_{L} \big(\delta _{(p,\omega)} \tens \delta _{(i,\xi)}^{(1)} \big) \tens a_{R}\big(\delta _{(i,\xi)}^{(2)} \tens \delta _{(k-p-i,\eta \omega^{-1}\xi^{-1})}\big)\cr 
&\quad+	\sum_{(p,\omega)} a_{L} \big(\delta _{(p,\omega)} \tens \delta _{(0,1)}^{(1)} \big) \tens a_{R}\big(\delta _{(0,1)}^{(2)} \tens \delta _{(k-p,\eta \omega^{-1})}\big)\cr 
&= \sum_{ (p,\omega),  (i,\xi) \neq(0, 1)} a_{L} \big(\delta _{(p,\omega)} \tens \delta _{(i,\xi)}^{(1)} \big) \tens a_{R}\big(\delta _{(i,\xi)}^{(2)} \tens (\delta _{(k-p-i,\eta \omega^{-1}\xi^{-1})} - \delta _{(k-p,\eta \omega^{-1})})\big)\cr 
&\quad+	\sum_{(p,\omega)} a_{L} \big(\delta _{(p,\omega)} \tens I_{m}\big) \tens a_{R}\big(I_{m} \tens \delta _{(k-p,\eta \omega^{-1})}\big)\ . 
\end{align*}
 Using $ a_{R}(F_{j k}\tens \delta_{(r,\xi)} )=  \dfrac{1}{n^{2}} \xi^{k-j} F_{j+r,k+r} $ and $ 	a_{L} (\delta_{(r,\xi)} \tens F_{jk}) =\dfrac{1}{n^{2}} \xi^{j-k} F_{j-r,k-r}  $ for both terms
\begin{align} \label{eq34}
	&= \dfrac{1}{n} 	\sum_{ (p,\omega),  (i,\xi) \neq(0, 1), j,a} a_{L}
	 \big(\delta _{(p,\omega)} \tens  \xi^{i-j+a} F_{a,j-i}        \big) \tens a_{R}\big(   F_{j,i+a} \tens (\delta _{(k-p-i,\eta \omega^{-1}\xi^{-1})} - \delta _{(k-p,\eta \omega^{-1})})\big)\cr 
	 	&= \dfrac{1}{n^{5}} 	\sum_{ (p,\omega),  (i,\xi) \neq(0, 1), j,a} 
	 \big(     \xi^{i-j+a} \omega^{a-j+i}  F_{a-p,j-i-p}       \big) \tens \cr &\quad \big(     (\eta \omega^{-1}\xi^{-1})^{i+a-j}     F_{j+k-p-i,i+a+k-p-i}   - (\eta \omega^{-1})^{i+a-j}  F_{j+k-p,i+a+k-p}\big)\cr  
	 	&= \dfrac{1}{n^{4}} 	\sum_{ p, j,a, (i,\xi) \neq(0, 1)} 
  (\xi \eta)^{i+a-j}  F_{a-p,j-i-p}   \tens \big(  \xi^{j-i-a}     F_{j+k-p-i,a+k-p}   -  F_{j+k-p,i+a+k-p}\big)\cr  
	 &= \dfrac{1}{n^{4}} 	\sum_{ p, j,a ,i\neq0, \xi  } 
	 \eta^{i+a-j}  F_{a-p,j-i-p}   \tens \big(    F_{j+k-p-i,a+k-p}   -    \xi^{i+a-j}   F_{j+k-p,i+a+k-p}\big)\cr
	 &\quad +\dfrac{1}{n^{4}} 	\sum_{ p, j,a , \xi \neq1 } 
	 \eta^{a-j}  F_{a-p,j-p}  \tens F_{j+k-p,a+k-p}  \big(    1  -    \xi^{a-j}  \big) \ .
\end{align}
The first part of equation (\ref{eq34}) is 
\begin{align} \label{eq35}
	&=  \dfrac{1}{n^{3}} 	\sum_{ p, j,a ,i\neq0 } 
	\eta^{i+a-j}  F_{a-p,j-i-p}   \tens \big(    F_{j+k-p-i,a+k-p}   -    \delta_{i+a-j,0}   F_{j+k-p,i+a+k-p}\big)\cr 
		&=  \dfrac{1}{n^{3}} 	\sum_{ p, j,a, i\neq0 } 
	\eta^{i+a-j}  F_{a-p,j-i-p} \tens   F_{j+k-p-i,a+k-p}     -   \dfrac{1}{n^{3}}   	\sum_{ p,a, i\neq0 } F_{a-p,a-p}  \tens  F_{i+a+k-p,i+a+k-p} \cr
		&=\dfrac{n-1}{n^{3}} 	\sum_{ p, s,a} 
		\eta^{a-s}  F_{a-p,s-p} \tens   F_{s+k-p,a+k-p}  -   \dfrac{1}{n^{3}}   	\sum_{ p,a, i\neq0 } F_{a-p,a-p}  \tens  F_{i+a+k-p,i+a+k-p} \cr
			&=   \dfrac{n-1}{n^{2}} 	\sum_{  b,r} 
			\eta^{b-r}  F_{b,r} \tens   F_{r+k,b+k}     -   \dfrac{1}{n^{2}}   	\sum_{b, i\neq0 } F_{b,b}  \tens  F_{i+k+b,i+k+b} 
\end{align}
where we have relabelled $ s=j-i $ in the first term. Next relabelling $ b=a-p $ and $ r=s-p $. Now let $ b=a-p $ and $ s=j-p $ in the second part of (\ref{eq34})
\begin{align} \label{eq38}
\dfrac{1}{n^{4}} 	\sum_{ p, s,b , \xi \neq1 } 
\eta^{b-s}  F_{b,s}  \tens F_{s+k,b+k}  \big(    1  -    \xi^{b-s}  \big)
&=\dfrac{1}{n^{3}} 	\sum_{  s,b , \xi \neq1 } 
\eta^{b-s}  F_{b,s}  \tens F_{s+k,b+k}  \big(    1  -    \xi^{b-s}  \big) \cr
&= \dfrac{1}{n^{2}} 	\sum_{  s,b } 
\eta^{b-s}  F_{b,s}  \tens F_{s+k,b+k}  \big(   1-\delta_{b,s}  \big)\ ,
\end{align}
since $ \displaystyle	\sum_{  \xi \neq1 } \big(    1  -    \xi^{b-s}  \big)=n(1-\delta_{b,s}) $, and adding equations (\ref{eq35}) and (\ref{eq38}) gives the results.
\end{proof}

\begin{prop}
	In Case 3
	 \begin{eqnarray}
	\omega^{\sharp} (\delta_{i}) =  \sum_{s \in \Z_{n}} 1_{,s} \tens 1_{,s-i}\ . 
	\end{eqnarray} 
\end{prop}
\begin{proof}
	From Theorem~\ref{integralHopfGalois} and (\ref{eq53}) 
\begin{eqnarray*}
\omega^{\sharp} (\delta_{i})= \sum_{j,k} a_{L} (\delta_{j}  \tens \delta_{k} {} ^{(1)})\tens a_{R} (\delta_{k}{} ^{(2)}  \tens \delta_{i-j-k} )= \sum_{j,k,t} a_{L} (\delta_{j}  \tens 1_{,t})\tens a_{R} (1_{,t-k}  \tens \delta_{i-j-k} )
\end{eqnarray*}
as $ \Delta_{R} b_{,t}= \displaystyle\sum_{s}b_{t+s} \tens \delta_{s}$, then
\begin{align*}
a_{L} (\delta_{j}  \tens 1_{,t})&=\sum_{s}b(\delta_{j} \tens \delta_{-s}) 1_{,t+s}= \sum_{s}(\int \delta_{j}  \delta_{s}) 1_{,t+s}= \dfrac{1}{n} 1_{,t+j}\ ,\cr
a_{R} (1_{,t} \tens \delta_{j}  )&=\sum_{s}  1_{,t+s} b(\delta_{s} \tens \delta_{j}) = \sum_{s}(\int  \delta_{s} \delta_{-j} )  1_{,t+s}= \dfrac{1}{n} 1_{,t-j}
\end{align*}
 so 
\begin{eqnarray}
\omega^{\sharp} (\delta_{i})=\dfrac{1}{n^{2}}  \sum_{j,k,t} 1_{,t+j} \tens 1_{,t-i+j}=\dfrac{1}{n}  \sum_{j,t} 1_{,t+j} \tens 1_{,t-i+j} 
\end{eqnarray} 
and setting $ s=t+j $ gives the answer.
\end{proof}

 \section{A Trivial Quantum Principle Bundle}\label{secTQPB}
 We first recall that if $\Phi,\Psi$ are maps from a coalgebra (in our case $H$) to an algebra then
so is the \textit{convolution product}\index{convolution}\index{$\conv$ convolution product}  $\conv$ defined by
\[ \Phi\conv \Psi=\cdot\,(\Phi\tens\Psi)\Delta\]
and that $\Phi$ is {\em{convolution-invertible}} when there is an inverse $ \Phi^{-1}$ such that
 $ \Phi\conv \Phi^{-1}= \Phi^{-1}\conv \Phi=1 . \epsilon$.  
 \begin{prop}\label{prop1}\cite{BMbook} Let $P$ be a right $H$-comodule algebra equipped with a convolution-invertible right-comodule map $\Phi:H\to P$ with $\Phi(1)=1$. Then $P$ is a quantum principal bundle over $A=P^H$. We call it a trivial bundle with trivialisation $\Phi$. 
 \end{prop}
We show that the algebra $ M_{n}(\C) $ is a trivial Hopf-Galois extension of the group $ \Z_{n} \times \Z_{n}$, referring to the construction in Section \ref{sec2} Case 2. We use the notation $ R_{n}$ to be the multiplicative group of the $ n $th roots of unity, generated by $ x=e^{2\pi i/n} $ with $ x^{n}=1 $.
\begin{prop}
The construction of Section \ref{sec2} Case 2 gives a trivial quantum principle bundle with algebra $ P=M_{n}(\C) $ and Hopf algebra $ H=\C [G] $ for $ G=\Z_{n}\times R_{n} $.
\end{prop}
\begin{proof}
	We start by defining $  \Phi(\delta_{(s,\omega)})=\displaystyle\sum_{ij} \Phi_{s,\omega,i,j} F_{ij}$ to be a right comodule map, meaning the following quantities are equal 
	\begin{align}
	\Delta_{R}\Phi (\delta_{(s,\omega )})&=\sum_{i,j,r,\xi}  \Phi_{s, \omega, i,j} \xi^{i-j} F_{i-r,j-r} \tens \delta_{(r,\xi)} \cr
	(\Phi \tens \id)\Delta (\delta_{(s,\omega )})&= \sum_{t,\eta} \Phi (\delta_{(t,\eta )}) \tens \delta_{(s-t,\frac{\omega}{\eta} )}=\sum_{t,\eta,p,q } \Phi _{t,\eta,p,q } F_{pq} \tens \delta_{(s-t,\frac{\omega}{\eta} )}\ .
	\end{align} 
	Using these we can show that there are $ \beta_{ij} $ with
	\begin{align*}
	\Phi_{s,\omega,i,j}= \omega^{j-i} \beta_{i-s,j-s}
	\end{align*}
 and for $ \Phi (1)=1  $ we need $ \displaystyle\sum_{i} \beta_{ii}=\frac{1}{n} $. The inverse $ \Psi $ of $ \Phi $ in Proposition \ref{prop1} can be shown to obey $ \Delta_{R} \Psi (h) = \Psi  (h_{(2)}) \tens Sh_{(1)}$. Writing $ \Psi  $ as $ \Psi (\delta _{(s,\omega)})=\displaystyle\sum_{ij} \Psi _{s, \omega, i ,j} F_{ij} $ we can show that $  \Psi _{s,\omega,i,j}= \omega^{i-j} \gamma_{i+s,j+s} $. The equations for convolution inverse reduce to 
	\begin{equation*}\label{aa}
	\eta^{j-i} \beta_{i-t,j-t} (\frac{\omega}{\eta})^{j-q} \gamma_{j+s-t,q+s-t} F_{i q}=\begin{cases}
	I_{n} & \mbox{if}  \ s=0 \text{ and} \ \omega=1 \\ 0 & \mbox{otherwise}  
	\end{cases} 
	\end{equation*}
	so 
	\begin{equation*}\label{bb}
	\sum_{t,\eta,j} \omega^{j-q} \eta^{q-i} \beta_{i-t,j-t}  \gamma_{j+s-t,q+s-t} =\begin{cases}
	\delta_{i q}& \mbox{if}  \ s=0 \text{ and} \ \omega=1 \\ & \mbox{otherwise}  
	\end{cases} 
	\end{equation*}
	the sum over $ \eta  $ gives zero unless $ q=i $, so we are left with 
	\begin{equation*}\label{d}
	n \sum_{j} \omega^{j-i} \sum_{t}\beta_{i-t,j-t}  \gamma_{j+s-t,i+s-t} =\begin{cases}
	1& \mbox{if} \ s=0 \text{ and} \ \omega=1 \\ 0  & \mbox{otherwise} \ .
	\end{cases} 
	\end{equation*}
	Now we use the result that evaluation gives an isomorphism from $ \C[\omega] _{<n}$ (the polynomials of degree $ <n \mod n $) to $ \C[R_{n}] $ the complex functions on the set $ R_{n} $. Using this we can rewrite (\ref{d}) as 
	\begin{equation*}\label{e}
	\sum_{t}\beta_{i-t,j-t}  \gamma_{j+s-t,i+s-t} =\begin{cases}
	\frac{1}{n} & \mbox{if} \ s=0  \\ 0 & \mbox{otherwise}
	\end{cases} 
	\end{equation*}
	let $ r=i-t $ so $ t=i-r $ and $ k=j-i $ then for $ k $ and $ s $
	\begin{equation}\label{g}
	\sum_{r}\beta_{r,r+k}  \gamma_{s+r+k,s+r} =\begin{cases}
	\frac{1}{n^{2}}  & \mbox{if}  \ s=0  \\ & \mbox{otherwise}\ .
	\end{cases} 
	\end{equation}
	The values $ \beta_{0,k}=\frac{1}{n} $, $ \gamma_{k,0}=\frac{1}{n} $ for all $ k\in \Z_{n} $ with all other $ \beta_{ij} $ and $ \gamma_{ij} $ zero solve (\ref{g}), and also $ \displaystyle\sum_{i} \beta_{ii}=\frac{1}{n} $. Thus we have a trivial Hopf-Galois extension. 
\end{proof}

\section{Consequences for Differential Calculus}
We can use Hopf-Galois extensions to study differential calculi on algebras. In particular we have the idea of a quantum principle bundle where we have the exact sequence 
\begin{eqnarray}\label{eq44} 
0 \longrightarrow P \, \Omega_{A}^{1} P  \xrightarrow{ \inc}\Omega_{P}^{1} \xrightarrow{\ver} P \tens \Lambda_{H}^{1} \longrightarrow 0
\end{eqnarray}
where $ \inc $ is the inclusion map and $ \Lambda_{H}^{1} $ is the left invariant $ 1 $-forms on $ H $. The vertical map is defined by 
\begin{eqnarray*}
\ver  (p . \extd q)=p q_{[0]} \tens  S(q_{[1]})  \extd q_{[2]}
\end{eqnarray*}
and is well defined if the map $ p . \extd q \mapsto p_{[0]} q_{[0]} \tens  p _{[1]} \extd q_{[1]}  $ from $\Omega_{P}^{1}   $ to $ P \tens \Omega_{H}^{1}  $ is well defined. This condition can be thought of as $ H $ coacting on $ P $ in a differentiable manner.

We shall use this theory to build calculi for matrices $ M_{n}(\C) $ from calculi on groups. We take the Case 2 of our previous discussion, where $ A= M_{1}(\C) $ and $ P=M_{n}(\C) $. As $ \extd 1=0 $, $ A $ must have the zero calculus, and (\ref{eq44}) becomes
\begin{eqnarray*}\label{eq45}
0 \longrightarrow 0  \longrightarrow \Omega_{P}^{1} \xrightarrow{\ver} P \tens \Lambda_{H}^{1} \longrightarrow 0
\end{eqnarray*}
so the left module map $\ver $ is an isomorphism. If we write $ \ver (\xi)=\xi^{0}\tens \xi^{1} $ then 
\begin{eqnarray}
\ver (p\, \xi )= p\,\xi^{0} \tens \xi^{1} \quad \text{ and} \quad \ver (\xi\, p  )= \xi^{0} p_{[0]}\tens \xi^{1} \triangleleft p_{[1]}
\end{eqnarray}
where $ \eta \triangleleft h =S(h_{(1)}) \eta h_{(2)}$. The first order left covariant differential calculi on $ H=\C[G]$ for a finite group $ G $
correspond to subsets $\mathcal{C}\subseteq G\setminus\{e\}$ \cite{BMbook}. The basis as a left module for the left invariant 1-forms is 
$e_a$ for $a\in\mathcal{C}$, with relations and exterior derivative for $f\in \C[G]$ being
\[ e_a.f=R_a(f)e_a,\quad \extd
f=\sum_{a\in {\mathcal{C}}}(R_a(f)-f) e_a\]
where $R_a(f)=f((\ )a)$ denotes right-translation. 
	We can calculate 
	\begin{align*}
	\ver(\extd \, p)&= \sum_{g\in G} g\triangleright p \tens S(\delta_{g_{(1)}}) \extd \delta _{g_{(2)}}\cr
	&= \sum_{g,x,y \in G :xy=g} \sum _{a\in \mathcal{C}} g\triangleright p \tens \delta_{x^{-1}} ( \delta_{y a^{-1}}-\delta_{y}  ) e_{a} \cr
	&=  \sum _{a\in \mathcal{C}} \sum_{x,g\in \mathcal{C}} g\triangleright p \tens \delta_{x^{-1}} (\delta_{g,a}-\delta_{g,e}) e_{a}=  \sum _{a\in \mathcal{C}} (a \triangleright p -p)\tens e_{a}\ .
	\end{align*}

Based on the idea of projective representation, we define a projective homomorphism $ \pi : G \rightarrow P^{\inv} $ for a group $ G $ and the group $ P^{\inv} $ of invertible elements of an algebra $ P $. If the algebra is over the field $ \C $ then we have $ \pi(e) =1$ and $ \pi(x)  \pi(y) =C _{x, y} \pi(x y) $ where $ e \in G $ is the group identity and for $ x,y \in G $ we have $ C _{x ,y} \in \C \setminus \{ 0\} $. Now we have an action of $ G $ on $ P $ by algebra maps $ g \triangleright p=  \pi (g) p  \pi(g)^{-1} $. For a finite group $ G $ we can write this as a right coaction of the Hopf algebra $ \C[G] $ of complex functions on $ G $, 
\begin{eqnarray}
 \Delta_{R}:P \rightarrow P \tens H , \quad \Delta_{R}(p)=p_{[0]}\tens p_{[1]}=\sum_{g\in G} g\triangleright p \tens \delta_{g}\ .
\end{eqnarray}
In the case where $ H $ has a left covariant calculus $ \Omega_{H}^{1} $, we have the right adjoint action of $ H $ on $ \Omega^{1}_{H} $
\[e_{a} \triangleleft \delta_{g}=\sum_{  x,y : xy=g} S(\delta_{x}) e_{a}\delta_{y}=\sum_{  x,y : xy=g} \delta_{x^{-1}} \delta_{ya^{-1}} e_{a}=\delta_{g,a} \sum_{x} \delta_{x} e_{a}=\delta_{g,a}e_{a}\ .\]
One of the simplest way to describe differential calculi is by using central generators (i.e. commute with elements of the algebra). We have an isomorphism $ \Omega_{P}^{1} \xrightarrow{\ver} P \tens \Lambda_{H}^{1} $, so a basis of $ \Lambda_{H}^{1} $ would generate $ \Omega_{P}^{1} $ as a left $ P $-module, but $ \Lambda_{H}^{1} $ is not central. However we can make another isomorphism in our Case 2, which will explicitly give $ \Omega_{P}^{1} $ for $ P=M_{n}(\C) $ in a frequently presented form with central generators. Consider the left $ P $- action on the image of $ \ver $, $ P \tens \Lambda_{H}^{1} $
\begin{eqnarray} \label{eq46}
(p \tens e_{a}) q =p q_{[0]} \tens e_{a} \triangleleft q_{[1]}=\sum_{g\in G} p (g\triangleright q) \tens e_{a} \triangleleft \delta_{g}=p(a \triangleright q) \tens e_{a} 
\end{eqnarray}
 in the case where the group action is given by a projective homomorphism we define a map $ \Phi: P \tens \Lambda_{H}^{1}  \rightarrow E= P \tens \Lambda_{H}^{1} $ by $  \Phi (p \tens e_{a})= p\, \pi (a) \tens  e_{a} $, then from (\ref{eq46}) 
\begin{eqnarray*}
\Phi ((p \tens e_{a}). q)=\Phi  (p \pi(a) q \pi(a)^{-1} \tens e_{a})=\Phi  (p \pi(a) q \tens e_{a}) =\Phi  (p\tens e_{a})\cdot (q \tens 1)
\end{eqnarray*}
using the product of tensor product. Thus on $ E= P \tens \Lambda_{H}^{1} $ the left and right actions are purely multiplication on the $ P $ part, i.e. $ \Lambda_{H}^{1}  $ is central. We can define $ \Omega_{P}^{1}= P \tens \Lambda_{H}^{1}$ with 
\begin{eqnarray*}
\extd p= \Phi ( \ver (\extd p))= \sum_{a \in \mathcal{C}} [\pi (a) , p] \tens e_{a}
\end{eqnarray*}
giving $ P $ a calculus with $ | \mathcal{C}    |  $ free central generators. In the case of $ P=M_{2}(\C) $ we have, taking $ \mathcal{C} $ to include all non identity elements, 
\begin{align}
\extd p&= \left[   \left( \begin{array}{cc}
1&0  \\
0&-1 
\end{array}
\right) , p   \right] \tens e_{g_{0,-1}} + \left[   \left( \begin{array}{cc}
0&1  \\
1&0 
\end{array}
\right) , p   \right] \tens e_{g_{1,1}} + \left[   \left( \begin{array}{cc}
0&1  \\
-1&0 
\end{array}
\right) , p   \right] \tens e_{g_{1,-1}}\cr
&= \left[  E_{00}- E_{11},p  \right]  \tens e_{g_{0,-1}}+ \left[  E_{01},p  \right]  \tens \left(e_{g_{1,1}}+ e_{g_{1,-1}} \right) + \left[  E_{10},p  \right]  \tens \left(e_{g_{1,1}}- e_{g_{1,-1}} \right)\ .
\end{align}
Thus the calculus for $ \mathcal{C}= \left\lbrace  (0,-1), (1,1), (1,-1) \right\rbrace  $ is the 3D universal calculus for $ M_{2} $, which in \cite{BMbook} Example 1.8 is defined in terms of the inner element $ \h (E_{00}-E_{11}) \oplus E_{01}\oplus E_{10} $. In addition we have $ \mathcal{C}= \left\lbrace   (1,1), (1,-1) \right\rbrace  $ corresponding to the 2D non-universal calculus with inner element $ E_{01}\oplus E_{10} $. Note that the structure of the group action we have imposed means that we do not get the full rage of calculi on $ M_{2} (\C) $ described in \cite{BMbook}.

\end{document}